\documentclass[12pt,oneside]{article}
\usepackage{amsmath,amsthm,amsfonts,amssymb,MnSymbol,mathrsfs,latexsym}
\usepackage{mathtools}
\usepackage[english]{babel}

\numberwithin{equation}{section}

\newcommand{\const}{\rm const}
  
  \newcommand{\supp}{\rm supp}

  \newcommand{\Dom}{\rm  Dom}

\theoremstyle{plain}
\newtheorem{theorem}{Theorem}[section]

\newtheorem{definition}{Definition}[section]
\newtheorem{remark}{Remark}[section]
\newtheorem{example}{Example}[section]

\renewenvironment{proof}{{\bf{Proof.}}}{\hfill $\Box$ \\}

\title{\large \textbf{Norm estimates in Grand Lebesgue Spaces for some operators, including magic square matrices}}

\footnotesize\date{}

\author{\normalsize Maria Rosaria Formica ${}^{1}$,   \normalsize Eugeny Ostrovsky
${}^2$ and \normalsize Leonid Sirota ${}^3$}

\begin{document}

  \maketitle

\begin{center}
{\footnotesize ${}^{1}$ Universit\`{a} degli Studi di Napoli \lq\lq Parthenope\rq\rq, via Generale Parisi 13,\\
Palazzo Pacanowsky, 80132,
Napoli, Italy.} \\

\vspace{1mm}

{\footnotesize e-mail: mara.formica@uniparthenope.it} \\

\vspace{2mm}

{\footnotesize ${}^{2,\, 3}$  Bar-Ilan University, Department of Mathematics and Statistics, \\
52900, Ramat Gan, Israel.} \\

\vspace{1mm}

{\footnotesize e-mail: eugostrovsky@list.ru}\\

\vspace{1mm}

{\footnotesize e-mail: sirota3@bezeqint.net} \\

\end{center}

\vspace{3mm}

\begin{abstract}
 \hspace{3mm} We extend the classical Lebesgue-Riesz norm estimations for integral operators acting between different classical Lebesgue-Riesz spaces into the Grand Lebesgue Spaces, in the general case. As an example we consider matrix operators acting between finite dimensional Lebesgue-Riesz spaces, especially generated by means of positive magic squares.

\end{abstract}

\vspace{3mm}


\vspace{3mm}

 \vspace{2mm}


 \vspace{5mm}

%

 \section{Introduction}

 \vspace{4mm}

\hspace{3mm}  Let $ \ (X = \{x\}, \cal{M}, \mu)  \ $ and  $ \ (Y = \{y\}, \cal{N}, \nu) \ $ be two measurable spaces with non - trivial measures $ \ \mu, \ \nu$, respectively. Let  $ \ p,q \in [1,\infty] \ $ and let $ \  f: X \to \mathbb R, \ g: Y \to \mathbb R \ $ be measurable numerical valued functions. The correspondent Lebesgue - Riesz spaces  and norms will be denoted, as usually, by $ \ L(p,X) = L_p, \ L(q,Y) = L_q, \ $

$$
||f||_{L(p,X)} = ||f||_p := \left[ \  \int_X |f(x)|^p \ \mu(dx)  \ \right]^{1/p},
$$

$$
|||g|||_{L(q,Y)} = |||g|||_q := \left[ \ \int_Y |g(y)|^q \ \nu(dy)  \ \right]^{1/q}.
$$

 \hspace{3mm} Let also $ \ U[\cdot] \ $ be a certain operator $ \ U: \ L(q,Y) \to L(p,X), \ $ not necessarily linear, such that there exists a constant $ \ \sigma \in (0, \infty) \ $ and two  non-empty intervals $ \ (a,b), \ 1 \le a < b \le \infty, \ (c,d), \ 1 \le c < d \le \infty \ $,
 for which

\vspace{3mm}

\begin{equation} \label{key condit}
\forall g \in L(q,Y) \ \Rightarrow \ ||U[g]||_p \le \sigma^{1/q - 1/p} \ |||g|||_q, \ p \in (a,b), \ q \in (c,d).
\end{equation}

\vspace{4mm}

 \hspace{3mm} On the other words, the operators norm of $ \ U \ $ acting from the space  $ \ L(q,Y) \ $ into the $  \ L(p,X), \ $  i.e. the value

 $$
 ||U||_{(L_q \to L_p)} \stackrel{def}{=} \sup_{g \ne 0} \left[ \  \frac{||U[g]||_p}{|||g|||_q}  \ \right]
 $$
allows the estimate

\begin{equation} \label{exact value}
 ||U||_{(L_q \to L_p)} \le  \sigma^{1/q - 1/p}, \  p \in (a,b), \ q \in (c,d)
\end{equation}
for some finite positive value $ \ \sigma, \ $ independent on $ \ (p,q). \ $ \par

 \ Note that the case when

\begin{equation} \label{exact up to const}
 ||U||_{(L_q \to L_p)} \le C \cdot  \sigma^{1/q - 1/p}, \  p \in (a,b), \ q \in (c,d), \ C = {\const} \in (0,\infty)
\end{equation}
may be investigated quite analogously. \par
 \ One can understood as a value $ \ C, \ $ of course, in the last relation its {\it minimal value:}

\vspace{3mm}

\begin{equation} \label{minvalue}
\underline{C} = \underline{C}(\sigma) \stackrel{def}{=} \sup_{g: \ |||g|||_q \in (0, \ \infty)} \ \sup_{p,q} \ \left\{ \sigma^{1/p - 1/q } \frac{||U[g]||_p}{|||g|||_q } \ \right\}.
\end{equation}

\vspace{4mm}

\hspace{3mm} {\bf  Our target in this short report is the extrapolation of the  inequality  (\ref{exact value}) into the so - called Grand Lebesgue
Spaces (GLS), instead of the classical Lebesgue - Riesz ones.  } \par

\vspace{4mm}

 \hspace{3mm}  {\it Some examples.} Put $ \ X = Y = \mathbb R^n, \ n = 1,2,\ldots. \ $ Let $ \ A: \mathbb R^n \to \mathbb R^n \ $ be a linear operator generated by  a square matrix
 with  the  correspondent  real valued entries $ \ a(i,j), \   i,j = 1,2,\ldots,n: \ $

$$
U[g](i) = \sum_{j=1}^n a(i,j)  g(j).
$$
 \ If in particular

$$
\sum_{j=1}^n |a(i,j)| \le 1, \ \ \forall i=1,\ldots,n
$$
 then  (\ref{exact value})  holds true with $ \ \sigma = 1. \ $ \par

 \vspace{3mm}

 \ Assume  now that the matrix $ \ A \ $ is a so - called  {\it magic square} matrix having only non - negative  entries:

\begin{equation}\label{def magic square}
\sum_j a(i,j) = \sum_i a(i,j) = \sum_k a(k,k) = \sum_l a(n-l +1,l) = \alpha = \const \in (0,\infty).
\end{equation}
 \  Both the measures $ \ \mu, \nu \ $ are uniform: \ $ \ \mu(j) = \nu(k)  = 1, \ j,k = 1,2,\ldots, n.  \ $  In the recent
article \cite{magic matr} it is proved that (\ref{exact up to const}) holds true with $ \ C = \alpha, \ \sigma = n. \ $ \par
 \ Moreover, in this case  one has

\begin{equation} \label{super exact}
 ||U||_{(L_q \to L_p)} = \alpha \cdot  n^{1/q - 1/p}, \  1 \le q \le p \le \infty,
\end{equation}
where $\alpha$ is the value given in \eqref{def magic square}.
\vspace{3mm}

 \ So, the value $ \ \alpha \ $ in (\ref{super exact}) is the best possible for all the values $ \ (p,q). \ $ \par

\vspace{3mm}

 \ Other examples may be found in   \cite{Ditzian,Nik1,Nik2}, devoting to the theory of approximation.\par

 \vspace{3mm}

 {\it We generalize  in this report the method offered in} \cite{Nik1}. \par

\vspace{4mm}

 \section{Grand Lebesgue Spaces.}

\vspace{4mm}

 \hspace{3mm}  We recall here, for the reader convenience, some known definitions and  facts  from the theory of Grand Lebesgue Spaces (GLS). \par

 \vspace{3mm}

    \ Let $ \ a = {\const} \ge 1, \  b = {\const} \in (a,\infty] $ and let $ \ \psi = \psi(p), \ p \in (a,b), \ $ be a positive measurable numerical valued function (real or complex), not necessarily finite in the boundary  points $ \  p = a + 0, \   p = b-0 \, $, such that $ \ \inf_{p \in [a, b)} \psi(p) > 0. \ $
The set of all such  functions  will be denoted as $ \ G \Psi(a,b)$; define also 
$$
G\Psi := \bigcup_{a, b} G \Psi(a,b).
$$


 For instance, functions belonging to $ G \Psi(a,b)$ are
$$
\psi_m(p) := p^{1/m}, \ \ m = {\const} > 0,  \ \ p \in [1,\infty)
$$
or, for $ 1  < b < \infty, \ \alpha,\beta = \const \ge 0$,
$$
   \psi_{b; \alpha,\beta}(p) := (p-1)^{-\alpha} \ (b-p)^{-\beta}, \  \ p \in
   (1,b).
$$

 \ The case $ \ m = 2 \ $ corresponds to the classical {\it subgaussian space}.\par

 The (Banach) Grand Lebesgue Space (GLS) \ $G \psi  = G\psi(a,b)$ on the space $X$ consists
of all the real (or complex) numerical valued measurable functions
$f: X \to \mathbb R$ having finite norm
\begin{equation} \label{norm psi}
    ||f||_{G\psi} \stackrel{def}{=} \sup_{p \in [a,b)} \left[ \frac{||f||_p}{\psi(p)} \right].
 \end{equation}

 \vspace{3mm}

The function $ \  \psi = \psi(p) \  $ is named {\it generating function } for the space $G \psi$.

If for instance
$$
  \psi(p) = \psi_{r}(p) = 1, \ \  p = r;  \ \  \ \psi_{r}(p) = +\infty,   \ \ p \ne r,
$$
 where $ \ r = {\const} \in [1,\infty),  \ C/\infty := 0, \ C \in \mathbb R, \ $ (an extremal case), then the correspondent
 $ \  G\psi_r  \  $ space coincides  with the classical Lebesgue - Riesz space $ \ L_r = L_r(X). \ $

\begin{remark}\label{classical grand Lebesgue}
 {\rm
 Let $1<q<\infty$ and $b=q$. Define $\psi(p)=(q-p)^{-1/p}$, \ $p\in (1,q)$, and let $X\subset \mathbb R^n$, \ $n\geq 1$, a measurable set with finite Lebesgue measure; then replacing in \eqref{norm psi} $p$ with $q-\varepsilon$, \ $\varepsilon\in(0,q-1)$, the space $G\psi$ reduces to 
 the classical Grand Lebesgue space $L^{q)}(X)$ defined by the norm
%
%
 \begin{equation*}
 ||f||_{L^{q)}(X)}=||f||_{q)}=\sup_{0<\varepsilon<q-1} \varepsilon^{\frac{1}{q-\varepsilon}}||f||_{q-\varepsilon}.
 \end{equation*}
 }
 \end{remark}

\vspace{4mm}

These spaces and their particular cases $L^{q)}(X)$ (Remark \ref{classical grand Lebesgue}) are investigated in many works (e.g.
\cite{Fiorenza2000,Kozachenko-Ostrovsky 1985,liflyandostrovskysirotaturkish2010,Ostrovsky 1994,Ostrov Prokhorov,Samko-Umarkhadzhiev,Samko-Umarkhadzhiev-addendum,anatriellofiojmaa2015,anatrielloformicaricmat2016}). For example
they play an important role in the theory of Partial Differential Equations (PDEs) (see, e.g., \cite{Ahmed-Fiorenza-Formica-Gogatishvili-Rakotoson,Fiorenza-Formica-Gogatishvili-DEA2018,fioformicarakodie2017,Iwaniec-Sbordone 1992,Greco-Iwaniec-Sbordone-1997}), in interpolation
theory (see, e.g.,
\cite{AFF2022_JFAA,fioforgogakoparakoNA,fiokarazanalanwen2004,formicagiovamjom2015}), in Functional Analysis (see, e.g., \cite{Fiorenza-Formica,FOS2021_Math_Nachr,FOS2021_JPDOA}), in the theory of
Probability (\cite{Ermakov etc. 1986,Kozachenko-Ostrovsky 1985,Ostrovsky1999,Ostrovsky HIAT,Ostrov Prokhorov,FOS2022_Contemporary Mathematics,ForKozOstr_Lithuanian}), in Statistics
(\cite{Kozachenko-Ostrovsky 1985,Ostrov Prokhorov,Ostrov frac der}, \cite[chapter 5]{Ostrovsky1999}), in theory of random fields
(\cite{Kozachenko-Ostrovsky 1985,Ostrovsky1999,CLT Holder sp}).

\vspace{3mm}

These spaces are rearrangement invariant (r.i.) Banach function
spaces; the fundamental function of the space $ \ G\psi \ $ has been studied in
\cite{Ostrovsky Fund fun} and has the form
\begin{equation} \label{Fund fun}
 \phi_{G\psi}(\delta)  \stackrel{def}{=} \sup_{p \in \Dom (\psi)} \left\{ \ \frac{\delta^{1/p}}{\psi(p)}  \ \right\}, \ \delta > 0.
\end{equation}
They not coincide, in the general case, with the
classical Banach rearrangement functional spaces: Orlicz, Lorentz,
Marcinkiewicz, etc., (see \cite{liflyandostrovskysirotaturkish2010,Ostrovsky HIAT}).

\hspace{3mm}  In the general case these spaces are non - separable. \par

The belonging of a measurable function $ f: X \to \mathbb{R}$ to some $ G\psi$
space is closely related to its tail function behavior
$$
 T_f(t) \stackrel{def}{=} {\bf P}(|f| \ge t)={\bf \mu}\{x \in X \ : \ |f(x)| \ge t\}, \ \ t \ge 0,
 $$
as $ \ t \to 0+ \ $ as well as when $ \ t \to \infty $ (see
\cite{Kozachenko-Ostrovsky 1985,KozOsSir2019,liflyandostrovskysirotaturkish2010,Ostrovsky 1994,Ostrovsky HIAT,Ostrov frac der,Ostrovsky1999,ForKozOstr_Lithuanian}) and so on.


\hspace{3mm} In detail, let $  \upsilon  $ be a non-zero measurable function belonging to some Grand Lebesgue Space $ \ G\psi \ $ and suppose  $ \ ||\upsilon||_{G\psi} = 1 $. Define the Young-Fenchel (or Legendre) transform of the function $\ h(p)=h[\psi](p)=p \ln \psi(p) \ $, 
 (see, e.g., \cite{Buld Koz AMS} and references therein)
$$
h^*(v)= h^*[\psi](v) := \sup_{p \in \Dom[\psi]}(p v - p \ln \psi(p))
$$
where $ \ \Dom[\psi] \ $ denotes the domain of definition (and finiteness) for the function $ \ \psi(\cdot). \ $

\vspace{1mm}
\ We get
\begin{equation} \label{Young Fen}
T_{\upsilon}(t) \le \exp(-h^*(\ln t)), \ \ t \ge e.
\end{equation}
 \ Note that the inverse conclusion, under suitable natural conditions, is true (see \cite{Ermakov etc. 1986,Kozachenko-Ostrovsky 1985,KozOsSir2019,liflyandostrovskysirotaturkish2010,Ostrovsky1999}). \par
 Notice also that these spaces coincide, up to equivalence of the norms and under appropriate conditions, with the so-called {\it exponential Orlicz} spaces, see, e.g.,
\cite{Ermakov etc. 1986,Kozachenko-Ostrovsky 1985,Ostrovsky1999,Ostrovsky HIAT}.

 \vspace{4mm}

  \ Let now  $ \ \zeta  \ $ be a measurable  function  $ \ \zeta: X \to \mathbb R  \ $  such that

$$
\exists a, b = {\const}, \ 1 \le a < b \le \infty \ : \ ||\zeta||_p < \infty, \ p \in (a,b).
$$
 \ The so-called {\it natural function} $ \ \psi[\zeta] = \psi[\zeta](p) \ $  for this variable is defined by

$$
\psi[\zeta](p) \stackrel{def}{=} ||\zeta||_p,
$$
with the corresponding domain of definition $ \ \Dom \ \psi[\zeta](\cdot). \ $ \par

Analogously, let $ \ \cal{V} \ $ be an arbitrary set and let $ \  \{\zeta_v \}, \ v \in \cal{V}, \ $  be a {\it family } of measurable functions such that
\begin{equation} \label{family}
\exists \, a,b = {\const} \in (1,\infty], \ a < b, \ : \ \sup_{v \in \cal{V}} ||\zeta_v||_p < \infty, \ \ \forall p \in [a,b)  \ .
\end{equation}
The {\it natural function} for this family is defined by
\begin{equation} \label{nat fun}
\psi[ \zeta_v](p) \stackrel{def}{=} \sup_{v\in \cal{V}} ||\zeta_v||_p,
\end{equation}
with the corresponding domain of definition $ \ {\Dom} \ \psi[\zeta_v]. \ $ \par
Evidently,
$$
\sup_{v \in \cal{V}} ||\zeta_v||_{G\psi[\zeta_v]} = 1.
$$

\vspace{4mm}

\section{Main result. Example.}

\subsection{Moment rearrangement invariant spaces. }

\vspace{4mm}

 \hspace{3mm} Let $ \ (W, || \cdot ||_W) \ $  be some rearrangement invariant (r.i.) space, equipped with the norm  $ || \cdot ||_W $, where $ \   W \ $ is a linear subset of the space of all
measurable functions $ \  f: T = \{t\}  \to \mathbb R  $  over a measurable space  $ \ (T, M, \mu) \ $.

\vspace{3mm}

\begin{definition}{\rm(see \cite{Nik1})}
{\rm
 We say that a space $ (W, || \cdot ||_W)$, equipped with the norm  $ || \cdot ||_W $, is a {\it moment rearrangement invariant space} (\emph{m.r.i.} space)
if there exist $ 1 \le a < b \le \infty $, and some {\it rearrangement invariant norm} $ \  < \cdot >, \ $ defined on the space of functions $h:(a,b)\to \mathbb R$, not necessarily finite on all the non-negative measurable functions, such that
\begin{equation} \label{mri}
\forall f \in W \ \Rightarrow \ ||f||_W \stackrel{def}{=} <h(\cdot)>, \hspace{3mm} h(p) = ||f||_p, \ \ p \in (a, b) .
\end{equation}

\vspace{3mm}

  \noindent We consider $(a,b)\stackrel{def}{=} {\supp}(W)$
(moment support, not necessarily uniquely defined).

\noindent Let $ 1 \le c < d \le \infty $ and define another m.r.i. space
 $ \ R = (R, || \cdot ||_R ) \ $
 such that
\begin{equation} \label{mri_R}
\forall f \in R \ \Rightarrow \ ||f||_R \stackrel{def}{=} <h(\cdot)>, \hspace{3mm} h(q) = ||f||_q, \ \ q \in (c, d) ,
\end{equation}
 with  $(c, d) \stackrel{def}{=} {\supp}(R) \ $. The following implication holds true:
\begin{equation} \label{implication}
 {\supp(W)} >> {\supp}(R ) \ \ \Longleftrightarrow \ \ \min(a, b) \ge \max(c, d).
\end{equation}
}
\end{definition}

\vspace{3mm}

 \begin{remark}
 {\rm
 Many examples of m.r.i. spaces with applications are provided in \cite{Nik1}.
 }
 \end{remark}

\vspace{4mm}

 \subsection{Main result.}

\vspace{4mm}

\begin{theorem}
Let $ 1 \le a < b \le \infty $,\  $ 1 \le c < d \le \infty $ and $p\in(a,b), \  q\in (c,d)$.
Suppose that the inequality
  \eqref{exact up to const} holds.  Let $ \psi, \ \nu$ be generating functions of the Grand Lebesgue spaces $G\psi(a,b) ,  \ G\nu(c,d)$, respectively .
 Then
\begin{equation} \label{main res}
\frac{||U[g]||_{G\psi}}{\phi_{G\psi} (\sigma^{-1})} \le \underline{C} \cdot \frac{||g||_{G\nu}}{\phi_{G\nu} (\sigma^{-1})},
\end{equation}
and the "constant" $ \ \underline{C} \ $ in  (\ref{main res}) is, in the general case, the best  possible. \par

\end{theorem}
\begin{proof}
The relation \eqref{exact up to const} may be rewritten as follows
\begin{equation} \label{key relat}
||U[g]||_p \ \sigma^{-1/q} \le C \ |||g|||_q \ \sigma^{-1/p}, \ \  p \in (a,b), \ q \in (c,d),
\end{equation}
and, dividing by $ \ \nu(q) \ $,
$$
||U[g]||_p \cdot \sigma^{-1/q}/\nu(q)  \le C \ \left(\ |||g|||_q/\nu(q) \ \right) \ \cdot \sigma^{-1/p}, \ \ p \in (a,b), \ q \in (c,d).
$$
We take now the supremum over $ \ q: \ $

$$
||U[g]||_p \cdot \phi_{G\nu}(1/\sigma) \le C \ ||g||_{G\nu} \cdot \sigma^{-1/p}.
$$
It remains to divide by $ \ \psi(p) \ $ and take the supremum over $ \ p: \ $
$$
||U[g]||_{G\psi} \cdot \phi_{G\nu}(1/\sigma) \le C \ ||g||_{G\nu} \cdot \phi_{G\psi} (1/\sigma),
$$
which is \eqref{main res}.

The exactness of the constant $ \ \underline{C} \ $ is grounded in particular in
\cite{Ostr int oper}.
\end{proof}

\vspace{4mm}

Consider now the m.r.i. spaces $W$ and $R$ introduced before. Taking from both sides of \eqref{key relat} the norm $ \ << \cdot \ >> $ and $ \ < \cdot > \ $ respectively, and taking into account the relation
$$
<<\delta^{1/p} >> = \phi_W(\delta), \ \  \delta > 0,
$$
the following  more general  estimate holds.
 \vspace{3mm}

 \begin{theorem}

\begin{equation} \label{aux space res}
\frac{<U[g]>}{\phi_W(\sigma^{-1})} \le \underline{C} \cdot \frac{<<g>>}{\phi_R(\sigma^{-1})},
\end{equation}
with non - improvable coefficient $ \ \underline{C}$, in the general case  .
\end{theorem}

The non-improvability is grounded in the particular case in the previous theorem.

\vspace{3mm}

\begin{example} \textbf{Magic square}.
{\rm
\vspace{3mm}

 \hspace{3mm} Let us return to the linear operator (\ref{super exact}) associated to the magic square matrix defined through \eqref{def magic square}. As long as there is therein the
 restriction $ \ q \le p, \ $ it is necessary to assume here that $ \  q \in (c,d), \ p \in (a,b), \ $ where
$$
1 \le c < d \le a < b \le \infty.
$$
Following, one can use \eqref{aux space res}  with  the exact value of  the constant $ \ \underline{C}  = \alpha$, with $\alpha$ defined in \eqref{def magic square}.
 }
 \end{example}

\vspace{6mm}

\emph{\textbf{\footnotesize Acknowledgements}.} {\footnotesize M.R. Formica is member of Gruppo
Nazionale per l'Analisi Matematica, la Probabilit\`{a} e le loro Applicazioni (GNAMPA) of the Istituto Nazionale di Alta Matematica (INdAM) and member of the UMI group \lq\lq Teoria dell'Approssimazione e Applicazioni (T.A.A.)\rq\rq and is partially supported by the INdAM-GNAMPA project, {\it Risultati di regolarit\`{a} per PDEs in spazi di funzione non-standard}, codice CUP\_E53C22001930001 and partially supported by University of Naples \lq\lq Parthenope\rq\rq, Dept. of Economic and Legal Studies, project CoRNDiS, DM MUR 737/2021, CUP I55F21003620001.
}

\end{document}